\documentclass[11pt]{article}
\usepackage{graphicx}
\usepackage{amsmath,amsfonts,amsthm}
\usepackage{amssymb,latexsym}
\usepackage{fixmath}
\usepackage{mathrsfs,amsbsy}
\usepackage{dsfont}
\usepackage{enumerate}
\usepackage{algorithm,algorithmic}
\usepackage{kbordermatrix}
\usepackage{xcolor}

\def\wtd{\widetilde}
\def\what{\widehat}

\DeclareMathOperator{\eig}{eig}

\DeclareMathOperator{\rank}{rank}

\DeclareMathOperator{\sv}{sv}

\DeclareMathOperator{\F}{F}
\DeclareMathOperator{\HH}{H}

\DeclareMathOperator{\UI}{ui}

\def\bbC{\mathbb{C}}

\def\bbR{\mathbb{R}}

\def\cR{{\cal R}}

\def\cX{{\cal X}}

\def\bg{\pmb{g}}

\def\bu{\pmb{u}}
\def\bv{\pmb{v}}

\def\bx{\pmb{x}}
\def\by{\pmb{y}}
\def\bz{\pmb{z}}

\def\bL{\pmb{L}}

\def\bT{\pmb{T}}

\def\bphi{\pmb{\phi}}

\def\scrB{\mathscr{B}}

\def\sss{\scriptscriptstyle}

\allowdisplaybreaks

\usepackage{hyperref}
\usepackage{cleveref}

\usepackage{accents}
\newcommand\munderbar[1]{%
  \underaccent{\bar}{#1}}

\DeclareMathOperator{\ext}{ext}

\newtheorem{theorem}{Theorem}[section]
\newtheorem{lemma}{Lemma}[section]
\newtheorem{corollary}{Corollary}[section]

\theoremstyle{definition}

\newtheorem{remark}{Remark}[section]

\def\sss{\scriptstyle}

\numberwithin{equation}{section}
\numberwithin{figure}{section}
\numberwithin{table}{section}

\newcounter{question}

\title{On Stewart's Perturbation Theorem for SVD}

\author{
Ren-Cang Li%
\thanks{Department of Mathematics, University of Texas at Arlington, Arlington, TX 76019-0408, USA.
        Supported in part by NSF  DMS-2009689.
        Email: {\tt rcli@uta.edu}.}
\and
Ninoslav Truhar%
\thanks{Department of Mathematics, J. J. Strossmayer University of Osijek, Trg Ljudevita Gaja 6, 31000 Osijek, Croatia.
        Email: {\tt ntruhar@mathos.hr}.}
\and
Lei-Hong Zhang%
\thanks{School of Mathematical Sciences, Soochow University, Suzhou 215006, Jiangsu, China.
        Supported in part by the National Natural Science Foundation of China NSFC-11671246
        NSFC-12371380, and  Jiangsu Shuangchuang Project JSSCTD202209.
        Email: {\tt longzlh@suda.edu.cn}.}
}

\date{\today}

\begin{document}

\maketitle

\begin{abstract}
This paper establishes a variant of Stewart's theorem [Theorem~6.4 of Stewart, {\em SIAM Rev.}, 15:727--764, 1973]
for the singular subspaces associated with the SVD of a matrix
subject to perturbations. Stewart's original version uses both the Frobenius and spectral norms,
whereas the new variant uses the spectral norm and any unitarily invariant norm that offer choices
per convenience of particular applications and lead to sharper bounds
than that straightforwardly derived from Stewart's original theorem with the help of the well-known equivalence inequalities
between matrix norms. Of interest in their own right, bounds on
the solution to two couple Sylvester equations are established for a few different circumstances.


\end{abstract}

\section{Introduction}\label{sec:intro}
In \cite{stew:1973}, Stewart established a perturbation theory
for the singular subspaces associated with the SVD of
a matrix $G\in\bbC^{m\times n}$ slightly perturbed. In the same paper he also investigated
the eigenspace of a matrix and the deflating subspace of a regular matrix pencil subject to perturbations.
But in this paper, we limit our scope
to one of his theorems, namely, \cite[Theorem~6.4]{stew:1973} on SVD, where both the Frobenius and spectral norms are used to measure perturbations.
Our goal is  to establish a more general and yet sharper version of \cite[Theorem~6.4]{stew:1973} in the spectral norm and any unitarily invariant norm.
Our new version
offers flexibility in its applications, for example, the version with the unitarily invariant norm also set to the spectral norm
is more convenient to use in our recent work \cite{zhli:2024} for
analyzing the quality of a reduced order model by approximate balanced truncation \cite{anto:2005,beoc:2017}.
Even in the case of using both the Frobenius and spectral norms, our results are slightly better than
Stewart's \cite[Theorem~6.4]{stew:1973} in terms of a less stringent condition and yet a sharper bound.

Given $G\in\bbC^{m\times n}$,
let
\begin{subequations}\label{eq:G:SVD:apx}
\begin{equation}\label{eq:UV}
U\equiv\kbordermatrix{ &\sss r &\sss m-r \\
                  & U_1 & U_2}\in\bbC^{m\times m}, \quad
V\equiv\kbordermatrix{ &\sss r &\sss n-r \\
                  & V_1 & V_2}\in\bbC^{n\times n}
\end{equation}
be two unitary matrices such that $G$ admits the decomposition
\begin{equation}\label{eq:G:SVD-1:apx}
U^{\HH}GV\equiv\begin{bmatrix}
        U_1^{\HH} \\
        U_2^{\HH}
      \end{bmatrix}G[V_1, V_2]
=\kbordermatrix{ &\sss r &\sss n-r \\
              \sss r &G_1 & 0 \\
              \sss m-r & 0 &G_2},
\end{equation}
\end{subequations}
where $1\le r<\min\{m,n\}$.
For example, this can be the SVD of $G$, for which
$G_1$ is diagonal with nonnegative diagonal entries as some of the singular values of $G$ and
$G_2$ is leading diagonal by which we mean that only its entries along the main diagonal line
starting at the top-left corner may be nonzero and nonnegative and they are some of the singular values of $G$, too.
By convention, $G$ has $\min\{m,n\}$ singular values $\sigma_i(G)$ arranged in decreasing order:
\begin{equation}\label{eq:sigma(G)}
\sigma_1(G)\ge\sigma_2(G)\ge\cdots\ge\sigma_{\min\{m,n\}}(G).
\end{equation}
Denote
the singular value set and its extended set of $G$ by
\begin{equation}\label{eq:svs(G)}
\sv(G)=\{\sigma_i(G)\}_{i=1}^{\min\{m,n\}}, \quad
\sv_{\ext}(G)=\sv(G)\cup\{\mbox{$|m-n|$ copies of $0$s}\},
\end{equation}
where the union is meant to be the multiset union.



Consider now that $G$ is perturbed to $\wtd G=G+E\in\bbC^{m\times n}$, and partition
\begin{equation}\label{eq:tG:apx-1}
U^{\HH}\wtd GV\equiv \begin{bmatrix}
        U_1^{\HH} \\
        U_2^{\HH}
      \end{bmatrix}(G+E)[V_1, V_2]
      =\kbordermatrix{ &\sss r &\sss n-r \\
              \sss r & G_1+E_{11} &  E_{12} \\
              \sss m-r &  E_{21} & G_2+E_{22}}.
\end{equation}
Naturally, one would ask whether $\wtd G$ admits a decomposition that is ``close'' to the one for $G$ in \eqref{eq:G:SVD:apx}.
Stewart \cite[Theorem~6.4]{stew:1973} provided an answer to that by
seeking orthogonal matrices  \cite[p.760]{stew:1973}
\begin{subequations}\label{eq:chkUV:apx}
\begin{align}
\check U&\equiv\kbordermatrix{ &\sss r &\sss m-r \\
                  & \check U_1 & \check U_2}
        =[U_1,U_2]\begin{bmatrix}
                    I_r & -\Gamma^{\HH} \\
                    \Gamma & I_{m-r}
                  \end{bmatrix}
                  \begin{bmatrix}
                    (I+\Gamma^{\HH}\Gamma)^{-1/2} & 0 \\
                    0 & (I+\Gamma\Gamma^{\HH})^{-1/2}
                  \end{bmatrix}, \label{eq:chkUV-1:apx}\\
\check V&\equiv\kbordermatrix{ &\sss r &\sss n-r \\
                  & \check V_1 & \check V_2}
        =[V_1,V_2]\begin{bmatrix}
                    I_r & -\Omega^{\HH} \\
                    \Omega & I_{n-r}
                  \end{bmatrix}
                  \begin{bmatrix}
                    (I+\Omega^{\HH}\Omega)^{-1/2} & 0 \\
                    0 & (I+\Omega\Omega^{\HH})^{-1/2}
                  \end{bmatrix} \label{eq:chkUV-2:apx}
\end{align}
\end{subequations}
such that
\begin{equation}\label{eq:SVD4tG-almost:apx}
\check U^{\HH}\wtd G\check V
\equiv \begin{bmatrix}
        \check U_1^{\HH} \\
        \check U_2^{\HH}
      \end{bmatrix}(G+E)[\check V_1, \check V_2]
  =\kbordermatrix{ &\sss r &\sss n-r \\
              \sss r & \check G_1 & 0 \\
              \sss m-r &  0 & \check G_2}.
\end{equation}

\Cref{thm:stew:1973SVD} below is \cite[Theorem~6.4]{stew:1973} in our notations here.

\begin{theorem}[{\cite[Theorem~6.4]{stew:1973}}]\label{thm:stew:1973SVD}
Given $G,\,\wtd G\in\bbC^{m\times n}$, let $G$ be decomposed as in \eqref{eq:G:SVD:apx}, and partition
$U^{\HH}\wtd GV$ according to \eqref{eq:tG:apx-1}. Let
\begin{subequations}\label{eq:stew:1973SVD-cond}
\begin{align}
\hat\varepsilon&=\sqrt{\|E_{12}\|_{\F}^2+\|E_{21}\|_{\F}^2}, \label{eq:stew:1973SVD-cond-1} \\
\delta&=\min_{\mu\in\sv(G_1),\,\nu\in\sv_{\ext}(G_2)}\,|\mu-\nu|, \label{eq:stew:1973SVD-cond-2} \\
\munderbar\delta&=\delta-\|E_{11}\|_2-\|E_{22}\|_2, \label{eq:stew:1973SVD-cond-3}
\end{align}
\end{subequations}
where $\|\cdot\|_2$ and $\|\cdot\|_{\F}$ denote the  spectral and Frobenius norms, respectively.
If
\begin{equation}\label{eq:stew:cond}
\munderbar\delta>0
\quad\mbox{and}\quad
\frac {\hat\varepsilon}{\munderbar\delta}<\frac 12,
\end{equation}
then there exist
$\Omega\in\bbC^{(n-r)\times r}$ and $\Gamma\in\bbC^{(m-r)\times r}$ satisfying
\begin{equation}\label{eq:stew:1973SVD-conc}
\sqrt{\|\Gamma\|_{\F}^2+\|\Omega\|_{\F}^2}
      \le\frac {1+\sqrt{1-4(\hat\varepsilon/\munderbar\delta)^2}}{1-2(\hat\varepsilon/\munderbar\delta)^2+\sqrt{1-4(\hat\varepsilon/\munderbar\delta)^2}}
         \frac {\hat\varepsilon}{\munderbar\delta}<2\,\frac {\hat\varepsilon}{\munderbar\delta}
\end{equation}
such that \eqref{eq:SVD4tG-almost:apx} with \eqref{eq:chkUV:apx} holds.
\end{theorem}

A number of results can be deduced from this informative theorem, e.g.,
bounds on $\|U-\check U\|_{\F}$ and $\|V-\check V\|_{\F}$, the explicit expressions of
$\check G_1$ and $\check G_2$ in terms of $\Omega$ and $\Gamma$, $G_1$, $G_2$, and $E_{ij}$, and also the singular values of $\wtd G$
is the multiset union of the singular values of $\check G_1$ and $\check G_2$.

Although the spectral norm $\|\cdot\|_2$ is used in defining $\munderbar\delta$,
the Frobenius norm $\|\cdot\|_{\F}$ is in full display in defining $\hat\varepsilon$ and in bounding $\Omega$ and $\Gamma$.
A straightforward version of it, using only the spectral norm can be easily derived in light of the equivalency inequalities
$$
\|B\|_2\le\|B\|_{\F}\le\sqrt{\rank(B)}\,\|B\|_2\le\sqrt{\min\{s,t\}}\,\|B\|_2
\quad\mbox{for $B\in\bbC^{s\times t}$}.
$$
For example, we can conclude that
{\em if $\munderbar\delta>0$ and
\begin{equation}\label{eq:stew:1973SVD-cond'}
\tilde\varepsilon:=\sqrt{\min\{m-r, n-r,r\}}\frac {\sqrt{\|E_{12}\|_2^2+\|E_{21}\|_2^2}}{\munderbar\delta}<\frac 12,
\end{equation}
then there exist
$\Omega\in\bbC^{(n-r)\times r}$ and $\Gamma\in\bbC^{(m-r)\times r}$ satisfying
\begin{equation}\label{eq:stew:1973SVD-conc'}
\sqrt{\|\Gamma\|_{\F}^2+\|\Omega\|_{\F}^2}
      \le\frac {1+\sqrt{1-4(\tilde\varepsilon/\munderbar\delta)^2}}{1-2(\tilde\varepsilon/\munderbar\delta)^2+\sqrt{1-4(\tilde\varepsilon/\munderbar\delta)^2}}
         \frac {\tilde\varepsilon}{\munderbar\delta}<2\,\frac {\tilde\varepsilon}{\munderbar\delta}
\end{equation}
such that \eqref{eq:SVD4tG-almost:apx} with \eqref{eq:chkUV:apx} holds.}
There are two clear drawbacks of such a straightforward version:
1) a much stronger condition in \eqref{eq:stew:1973SVD-cond'} than something like
\begin{equation}\label{eq:stew:1973SVD-cond''}
\frac {\sqrt{\|E_{12}\|_2^2+\|E_{21}\|_2^2}}{\munderbar\delta}<\frac 12
\end{equation}
as one might possibly expect, and 2) a much weaker
conclusion in \eqref{eq:stew:1973SVD-conc'} than
something like
\begin{equation}\label{eq:stew:1973SVD-conc''}
\sqrt{\|\Gamma\|_2^2+\|\Omega\|_2^2}\le
\sqrt{\|\Gamma\|_{\F}^2+\|\Omega\|_{\F}^2}< 2\,\frac {\sqrt{\|E_{12}\|_2^2+\|E_{21}\|_2^2}}{\munderbar\delta}
\end{equation}
also as one might possibly expect. The appearances of $m$ and $n$ in \eqref{eq:stew:1973SVD-cond'}
are particularly unsatisfactory for large (huge) $m$ and usually modest $r$.

Our goal in this paper is to create a version of \Cref{thm:stew:1973SVD} that  uses
the spectral norm and any unitarily invariant norm. Our eventual version, \Cref{thm:SVD-almost},
when restricted exclusively to the spectral norm,  does not  exactly coincide with the possible expectations in
\eqref{eq:stew:1973SVD-cond''} and \eqref{eq:stew:1973SVD-conc''}, but is very much close to it, namely
\eqref{eq:stew:1973SVD-cond''} and \eqref{eq:stew:1973SVD-conc''} with
$\sqrt{\|E_{12}\|_2^2+\|E_{21}\|_2^2}$ and $\sqrt{\|\Gamma\|_2^2+\|\Omega\|_2^2}$ replaced by
$\max\{\|E_{12}\|_2,\|E_{21}\|_2\}$ and $\max\{\|\Gamma\|_2,\|\Omega\|_2\}$, respectively.
In particular, $m$ and $n$ disappear altogether.

The rest of this paper is organized as follows. \Cref{sec:prelim} states two preliminary results for later use in our proofs.
Our main result, a variant of Stewart's, is stated and proved in \Cref{sec:main}.
In \Cref{sec:disscuss}, we compare our main result realized for the spectral norm only and for both the spectral
and Frobenius norm with Stewart's result in \Cref{thm:stew:1973SVD} and its potential consequences
in \eqref{eq:stew:1973SVD-cond'} and \eqref{eq:stew:1973SVD-conc'}.
Finally, we draw our conclusions in \Cref{sec:concl}.

\noindent
{\bf Notation}.
$\bbC^{n\times m}$ is the set
of all $n\times m$ complex matrices, $\bbC^n=\bbC^{n\times 1}$,
and $\bbC=\bbC^1$.
$I_n$ (or simply $I$ if its dimension is
clear from the context) is the $n\times n$ identity matrix.
The superscript $X^{\HH}$ is the complex conjugate transpose of a matrix or vector $X$.
We shall also adopt MATLAB-like convention to
access the entries of vectors and matrices.
Let $i:j$ be  the set of integers from $i$ to $j$ inclusive.
$X_{(k:\ell,i:j)}$, $X_{(k:\ell,:)}$, and $X_{(:,i:j)}$ are submatrices of $X$,
consisting of intersections of row $k$ to row $\ell$ and  column $i$ to column $j$,
row $k$ to row $\ell$, and column $i$ to column $j$, respectively.
${\cal R}(B)$ is the column space of $B$, i.e., the subspace spanned by the columns of $B$.
Finally $\eig(A)$ is the spectrum of a square matrix $A$.

We will continue to adopt the notation introduced so far in this section such as $\|\cdot\|_2$ and $\|\cdot\|_{\F}$.
In particular, $G\in\bbC^{m\times n}$ has $\min\{m,n\}$ singular values denoted by $\sigma_i(G)$ in decreasing order as in
\eqref{eq:sigma(G)}, and accordingly two singular value sets $\sv(G)$ and $\sv_{\ext}(G)$ in \eqref{eq:svs(G)}, and
$\sigma_{\max}(G):=\sigma_1(G)$ and $\sigma_{\min}(G):=\sigma_{\min\{m,n\}}(G)$.

\section{Preliminaries}\label{sec:prelim}
In this section, we will state four lemmas that we will need later.
The first lemma, \Cref{lm:stew-lm}, is due to Stewart \cite{stew:1971,stew:1973}, and
the second one, \Cref{lm:BdSylv}, summarizes known bounds on the solution to the Sylvester equation with Hermitian coefficient matrices
of Davis and Kahan~\cite{daka:1970} and of Bhatia, Davis, and McIntosh~\cite{bhdm:1983}.
The third one, \Cref{lm:BdCpldSylv}, establishes
bounds on the solution pair to a set of the coupled Sylvester equations and the results are new, except the one for
the Frobenius norm. Finally, the fourth lemma, \Cref{lm:SVr}, is likely known.


\begin{lemma}[{\cite[Theorem 3.5]{stew:1971}, \cite[Theorem 3.1]{stew:1973}}]\label{lm:stew-lm}
Let $\bT$ be a bounded linear operator on the Banach space $(\scrB,\|\cdot\|)$ that
has a bounded inverse. Let $\bphi$ be a continuous function on $(\scrB,\|\cdot\|)$
that satisfies, for $\bx,\,\by\in\scrB$,
\begin{align*}
  \mbox{\rm (i)}  &\quad \|\bphi(\bx)\|\le\hat\varepsilon\|\bx\|^2, \\
  \mbox{\rm (ii)} &\quad \|\bphi(\bx)-\bphi(\by)\|\le 2\hat\varepsilon\max\{\|\bx\|,\|\by\|\}\|\bx-\by\|,
\end{align*}
for some $\hat\varepsilon\ge 0$. Let
$
0<\hat\delta\le\|\bT^{-1}\|^{-1}.
$
Given $\bg\in\scrB$, if
$$
\kappa_2:=(\hat\varepsilon/\hat\delta^2)\|\bg\|<1/4,
$$
then equation
$
\bT\bx=\bg-\bphi(\bx)
$
has a solution $\bx\in\scrB$ that satisfies
$$
\|\bx\|\le 
   \frac {1+\sqrt{1-4\kappa_2}}{1-2\kappa_2+\sqrt{1-4\kappa_2}}\frac {\|\bg\|}{\hat\delta}
   < 2\frac {\|\bg\|}{\hat\delta}.
$$
\end{lemma}

We need the notation of unitarily invariant norm to go forward. A matrix norm $\|\cdot\|_{\UI}$ is called a {\em unitarily invariant norm\/} on $\bbC^{m\times n}$ if it is a matrix norm
and has the following two additional properties \cite{bhat:1996,stsu:1990}:
\begin{enumerate}
  \item $\| U^{\HH}BV\|_{\UI}=\| B\|_{\UI}$ for all unitary matrices $U\in\bbC^{m\times m}$ and $V\in\bbC^{n\times n}$
        and $B\in\bbC^{m\times n}$;
  \item $\| B\|_{\UI}=\|B\|_2$, the spectral norm of $B$, if $\rank(B)=1$.
\end{enumerate}
Two commonly used unitarily invariant norms are
$$
\begin{tabular}{rl}
the spectral norm: & $\|B\|_2=\max_j\sigma_j$,  \\
the Frobenius norm: & $\|B\|_{\F}=\sqrt{\sum_j\sigma_j^2}$\,,
\end{tabular}
$$
where $\sigma_1,\sigma_2,\ldots,\sigma_{\min\{m,n\}}$ are the singular values of $B$.
In what follows, $\|\cdot\|_{\UI}$ denotes a general unitarily invariant norm.


In this article, for convenience, any $\|\cdot\|_{\UI}$
we use is generic to matrix sizes in the sense that it applies to matrices of all sizes.
Examples include the matrix spectral norm $\|\cdot\|_2$, the Frobenius norm $\|\cdot\|_{\F}$, and the trace norm. One important property
of unitarily invariant norms is
$$
\| XYZ\|_{\UI}\le\|X\|_2\cdot\| Y\|_{\UI}\cdot\|Z\|_2
$$
for any matrices $X$, $Y$, and $Z$ of compatible sizes.

The next lemma summarizes known bounds on the solution of the Sylvester equation $AX-XB=S$. These bounds have played important roles
in eigenspace variations as demonstrated by Davis and Kahan~\cite[1970]{daka:1970}, and Bhatia, Davis, and McIntosh~\cite[1983]{bhdm:1983}.
For a brief review, see \cite{li:1999d}.

\begin{lemma}\label{lm:BdSylv}
Consider matrix equation $XA-BX=S$, where  $A\in\bbR^{r\times r}$ and $B\in\bbR^{s\times s}$ are Hermitian, and $S\in\bbC^{s\times r}$.
If $\eig(A)\cap\eig(B)=\emptyset$, then the equation has a unique solution $X\in\bbC^{s\times r}$. Furthermore, the following statements
hold.
\begin{enumerate}[{\rm (a)}]
  \item {\rm \cite{daka:1970}} we have
        \begin{equation}\label{ineq:DK-F}
        \|X\|_{\F}\le\|S\|_{\F}/\delta,\quad
        \delta:=\min_{\mu\in\eig(A),\,\nu\in\eig(B)}\,|\mu-\nu|;
        \end{equation}
  \item {\rm \cite{daka:1970}} if there exist $\alpha<\beta$ and $\delta>0$ such that
        \begin{equation}\label{eq:stric-sepd}
        \begin{array}{rl}
        \mbox{either}\quad&\mbox{$\eig(A)\subset [\alpha,\beta]$ and $\eig(B)\subset(-\infty,\alpha-\delta]\cup [\beta+\delta,\infty)$}, \\
        \mbox{or}\quad&\mbox{$\eig(A)\subset (-\infty,\alpha-\delta]\cup [\beta+\delta,\infty)$ and $\eig(B)\subset[\alpha,\beta]$},
        \end{array}
        \end{equation}
        then for any unitarily invariant norm $\|\cdot\|_{\UI}$,
        \begin{equation}\label{ineq:DK-UI}
        \|X\|_{\UI}\le\|S\|_{\UI}/\delta;
        \end{equation}
  \item {\rm \cite{bhdm:1983,bhro:1997}} we have
        \begin{equation}\label{ineq:BDM-UI}
        \|X\|_{\UI}\le(\pi/2)\|S\|_{\UI}/\delta,
        \end{equation}
        where $\delta$ is as in \eqref{ineq:DK-F}.
\end{enumerate}
\end{lemma}

The results of \Cref{lm:BdSylv} have an alternative interpretation through a linear operator
\begin{equation}\label{eq:L(X):defn}
\bL\,:\,X\in\bbC^{s\times r}\,\to\, \bL(X)=XA-BX\in\scrB,
\end{equation}
endowed with certain unitarily invariant norm  $\|\cdot\|_{\UI}$ on $\bbC^{s\times r}$, including
the spectral and Frobenius norms as special ones, where $A\in\bbR^{r\times r}$ and $B\in\bbR^{s\times s}$ are Hermitian.
With any given $\|\cdot\|_{\UI}$, there is an induced operator norm $\|\cdot\|$ on the linear operator from $\bbC^{s\times r}$ to itself.
Translating the results of \Cref{lm:BdSylv} yields the following corollary.

\begin{corollary}\label{cor:BdSylv}
Let $A\in\bbR^{r\times r}$ and $B\in\bbR^{s\times s}$ be Hermitian, and define linear operator $\bL$
on $(\bbC^{s\times r},\|\cdot\|_{\UI})$ as in \eqref{eq:L(X):defn}.
If $\eig(A)\cap\eig(B)=\emptyset$, then $\bL$ is invertible, and, furthermore, the following statements
hold.
\begin{enumerate}[{\rm (i)}]
  \item With $\|\cdot\|_{\UI}=\|\cdot\|_{\F}$ and $\delta$ as in \eqref{ineq:DK-F}, we have
        $\|\bL^{-1}\|^{-1}=\delta$;
  \item With general $\|\cdot\|_{\UI}$ and assuming \eqref{eq:stric-sepd},
        then $\|\bL^{-1}\|^{-1}=\delta$,
        where $\delta$ is the largest $|\mu-\nu|$ for some $\mu\in\eig(A),\,\nu\in\eig(B)$ and
        subject to \eqref{eq:stric-sepd};
  \item With general $\|\cdot\|_{\UI}$  and $\delta$ as in \eqref{ineq:DK-F},
        we have
        $\|\bL^{-1}\|^{-1}\ge (2/\pi)\delta$.
\end{enumerate}
\end{corollary}

\begin{proof}
To see these, we note that $\|\bL^{-1}\|^{-1}=\min\{\gamma\,:\,\|\bL(X)\|\ge\gamma\|(X)\|\}$. Hence,
immediately it follows that $\|\bL^{-1}\|^{-1}\ge\delta$ for items (i) and (ii) and
$\|\bL^{-1}\|^{-1}\ge (2/\pi)\delta$ for item (iii). The equality sign in items (i) and (ii)
are achieved by letting $X=\bx\by^{\HH}$, where $\bx,\,\by$ are unit
eigenvectors of $A$ and $B$, respectively, such that
$A\bx=\mu\,\bx$ and $B\by=\nu\,\by$, assuming $\delta=|\mu-\nu|$, and hence
$$
\bL(X)=(\mu-\nu)\,\bx\by^{\HH}=(\mu-\nu) X.
$$
Also because $\sv(X)$ consists of one nonzero singular value $1$ and some copies of zeros,
$$
\|X\|_{\UI}=\|X\|_2=1,
$$
implying $\|\bL(X)\|_{\UI}=\delta\|X\|_{\UI}$ for item~(i) or item~(ii).
\end{proof}

For subspace variation associated with SVD, a set of two coupled Sylvester equations come into play:
\begin{equation}\label{eq:CpldSyvl}
XA-BY=S,\quad YA^{\HH}-B^{\HH}X=T,
\end{equation}
where $A\in\bbC^{r\times r}$, $B\in\bbC^{s\times t}$, $S\in\bbC^{s\times r}$, and $T\in\bbC^{t\times r}$.
Note that $B$ is possibly a nonsquare matrix, in which case we pad some blocks of zeros to
$B$, $Y$, and $S$ or $T$ to yield an equivalent set of two coupled Sylvester equations with $B$ being square.
Specifically, if $s>t$, let
$$
B_{\ext}=[B,0_{s\times (s-t)}],\,
Y_{\ext}=\begin{bmatrix}
     Y \\
     0_{(s-t)\times r}
   \end{bmatrix},\,
S_{\ext}=S,\,\,
T_{\ext}=\begin{bmatrix}
     T \\
     0_{(s-t)\times r}
     \end{bmatrix},
$$
whereas if $s<t$, let
$$
B_{\ext}=\begin{bmatrix}
     B \\
     0_{(t-s)\times t}
   \end{bmatrix},\,
Y_{\ext}=[Y, 0_{(t-s)\times r}],\,
S_{\ext}=\begin{bmatrix}
     S \\
     0_{(t-s)\times r}
     \end{bmatrix},\,\,
T_{\ext}=T.
$$
Then \eqref{eq:CpldSyvl} is equivalent to
\begin{equation}\label{eq:CpldSyvl-ext}
XA-B_{\ext}Y_{\ext}=S_{\ext},\quad Y_{\ext}A^{\HH}-B_{\ext}^{\HH}X=T_{\ext},
\end{equation}
which can be merged into one Sylvester equation
\begin{equation}\label{eq:CpldSyvl'}
\begin{bmatrix}
  0 & X \\
  Y_{\ext} & 0
\end{bmatrix}\begin{bmatrix}
               0 & A^{\HH} \\
               A & 0
             \end{bmatrix}
-\begin{bmatrix}
   0 & B_{\ext} \\
   B_{\ext}^{\HH} & 0
 \end{bmatrix}\begin{bmatrix}
  0 & X \\
  Y_{\ext} & 0
\end{bmatrix}=\begin{bmatrix}
                S_{\ext} & 0 \\
                0 & T_{\ext}
              \end{bmatrix}.
\end{equation}
It can be seen that
\begin{equation}\label{eq:eig=sv}
\eig(\begin{bmatrix}
               0 & A^{\HH} \\
               A & 0
             \end{bmatrix})=\sv(A)\cup(-\sv(A)), \,
\eig(\begin{bmatrix}
   0 & B_{\ext} \\
   B_{\ext}^{\HH} & 0
 \end{bmatrix})=\sv_{\ext}(B)\cup(-\sv_{\ext}(B)),
\end{equation}
where negating a set means negating each element of the set, and
\begin{equation}\label{eq:UI-invariant}
\left\|\begin{bmatrix}
  0 & X \\
  Y_{\ext} & 0
\end{bmatrix}\right\|_{\UI}
=\left\|\begin{bmatrix}
  0 & X \\
  Y & 0
\end{bmatrix}\right\|_{\UI}, \quad
\left\|\begin{bmatrix}
                S_{\ext} & 0 \\
                0 & T_{\ext}
              \end{bmatrix}\right\|_{\UI}
=\left\|\begin{bmatrix}
                S & 0 \\
                0 & T
              \end{bmatrix}\right\|_{\UI}
\end{equation}
for any unitary invariant norm. Apply \Cref{lm:BdSylv} to get

\begin{lemma}\label{lm:BdCpldSylv}
Consider a set of two coupled Sylvester equations \eqref{eq:CpldSyvl}, where $A\in\bbC^{r\times r}$, $B\in\bbC^{s\times t}$, $S\in\bbC^{s\times r}$, and $T\in\bbC^{t\times r}$.
If $\sv(A)\cap\sv_{\ext}(B)=\emptyset$, then the set of equations has a unique solution pair
$(X,Y)\in\bbC^{s\times r}\times\bbC^{t\times r}$. Furthermore, the following statements
hold:
\begin{enumerate}[{\rm (a)}]
  \item  we have {\rm \cite{stew:1973}}
        \begin{equation}\label{ineq:stew-F}
        \sqrt{\|X\|_{\F}^2+\|Y\|_{\F}^2}\le\left.\sqrt{\|S\|_{\F}^2+\|T\|_{\F}^2}\right/\delta,\quad
        \delta:=\min_{\omega\in\sv(A),\,\gamma\in\sv_{\ext}(B)}\,|\omega-\gamma|;
        \end{equation}
  \item if $\delta:=\sigma_{\min}(A)-\sigma_{\max}(B)>0$,
        then for any unitarily invariant norm $\|\cdot\|_{\UI}$,
       \begin{subequations}\label{ineq:CpldSylv-UI}
       \begin{align}
        \left\|\begin{bmatrix}
          0 & X \\
          Y & 0
        \end{bmatrix}\right\|_{\UI}&\le\frac 1{\delta}\left\|\begin{bmatrix}
                        S & 0 \\
                        0 & T
                      \end{bmatrix}\right\|_{\UI}, \label{ineq:CpldSylv-UI-1}\\
       \max\{\|X\|_{\UI},\|Y\|_{\UI}\}&\le\frac 1{\delta}\max\{\|S\|_{\UI},\|T\|_{\UI}\},  \label{ineq:CpldSylv-UI-2}
       \end{align}
       \end{subequations}
        and in particular for the spectral norm
        \begin{equation}\label{ineq:CpldSylv-2}
        \max\{\|X\|_2,\|Y\|_2\}\le \frac 1{\delta}\max\{\|S\|_2,\|T\|_2\};
        \end{equation}
  \item we have
        \begin{subequations}\label{ineq:CpldSylv-UI'}
        \begin{align}
        \left\|\begin{bmatrix}
          0 & X \\
          Y & 0
        \end{bmatrix}\right\|_{\UI}&\le\frac {\pi}2\frac 1{\delta}\left\|\begin{bmatrix}
                        S & 0 \\
                        0 & T
                      \end{bmatrix}\right\|_{\UI}, \label{ineq:CpldSylv-UI'-1} \\
       \max\{\|X\|_{\UI},\|Y\|_{\UI}\}&\le\pi\frac 1{\delta}\max\{\|S\|_{\UI},\|T\|_{\UI}\}, \label{ineq:CpldSylv-UI'-2}
        \end{align}
        \end{subequations}
        where $\delta$ is as in \eqref{ineq:stew-F},
        and in particular for the spectral norm
        \begin{equation}\label{ineq:CpldSylv-2'}
        \max\{\|X\|_2,\|Y\|_2\}\le \frac {\pi}2\frac 1{\delta}\max\{\|S\|_2,\|T\|_2\}.
        \end{equation}
\end{enumerate}
\end{lemma}

\begin{proof}
Recall \eqref{eq:eig=sv} and \eqref{eq:UI-invariant}.
The inequality in \eqref{ineq:stew-F} is essentially Stewart's \cite[Theorem~6.2]{stew:1973}, but here as a corollary
of \Cref{lm:BdSylv} applied to \eqref{eq:CpldSyvl'}.
Inequalities \eqref{ineq:CpldSylv-UI-1} and \eqref{ineq:CpldSylv-UI'-1}
are also corollaries
of \Cref{lm:BdSylv} applied to \eqref{eq:CpldSyvl'}.
Inequalities \eqref{ineq:CpldSylv-2} and \eqref{ineq:CpldSylv-2'} follow
from \eqref{ineq:CpldSylv-UI} and \eqref{ineq:CpldSylv-UI'-1}, respectively, due to
\begin{equation}\label{eq:UI=2:equiv}
\left\|\begin{bmatrix}
  0 & X \\
  Y & 0
\end{bmatrix}\right\|_2=\max\{\|X\|_2,\|Y\|_2\}, \quad
\left\|\begin{bmatrix}
                        S & 0 \\
                        0 & T
                      \end{bmatrix}\right\|_2=\max\{\|S\|_2,\|T\|_2\}.
\end{equation}
It remains to show \eqref{ineq:CpldSylv-UI-2} and \eqref{ineq:CpldSylv-UI'-2}.
Inequality \eqref{ineq:CpldSylv-UI-2} is essentially implied in \cite[section 3]{wedi:1972}, but we will present
a quick proof anyway.
Note that for the case
$\delta:=\sigma_{\min}(A)-\sigma_{\max}(B)>0$.
We have
\begin{subequations}\label{eq:4bT:pf-2}
\begin{align}
\|X A-BY\|_{\UI}
 &\ge \|X A\|_{\UI}-\|BY\|_{\UI}\nonumber \\
 &\ge \|X\|_{\UI}\|A^{-1}\|_2^{-1}-\|B\|_2\|Y\|_{\UI}\nonumber \\
 &=\|X\|_{\UI}\sigma_{\min}(A)-\sigma_{\max}(B)\|Y\|_{\UI}. \label{eq:4bT1}
\end{align}
Similarly, we can get
\begin{equation}\label{eq:4bT2}
\|Y A^{\HH}-B^{\HH}X\|_{\UI}
  \ge\sigma_{\min}(A)\|Y\|_{\UI}-\sigma_{\max}(B)\|X\|_{\UI}.
\end{equation}
\end{subequations}
There are two cases to consider. If $\|X\|_{\UI}\ge\|Y\|_{\UI}$, then by \eqref{eq:4bT1} we get
\begin{subequations}\label{eq:4bT:pf-3}
\begin{equation}\label{eq:T:case1}
\|X A-BY\|_{\UI}
   \ge\delta\|X\|_{\UI}=\delta\,\max\{\|X\|_{\UI},\|Y\|_{\UI}\};
\end{equation}
if, on the other hand, $\|X\|_{\UI}<\|Y\|_{\UI}$, then by \eqref{eq:4bT2} we get
\begin{equation}\label{eq:T:case2}
\|Y A^{\HH}-B^{\HH}X\|_{\UI}
   \ge\delta\|Y\|_{\UI}=\delta\,\max\{\|X\|_{\UI},\|Y\|_{\UI}\}.
\end{equation}
\end{subequations}
Together, \eqref{eq:T:case1} and \eqref{eq:T:case2} yield
\begin{align*}
\max\left\{\|X A-BY\|_{\UI},
                               \|Y A^{\HH}-B^{\HH}X\|_{\UI}\right\}
      \ge\delta\,\max\{\|X\|_{\UI},\|Y\|_{\UI}\},
\end{align*}
as expected. Finally, noticing that
\begin{align*}
\max\{\|X\|_{\UI},\|Y\|_{\UI}\}
   &\le\left\|\begin{bmatrix}
          0 & X \\
          Y & 0
        \end{bmatrix}\right\|_{\UI}, \\
\left\|\begin{bmatrix}
                        S & 0 \\
                        0 & T
                      \end{bmatrix}\right\|_{\UI}
   &\le\|S\|_{\UI}+\|T\|_{\UI} \\
   &\le 2\max\{\|S\|_{\UI},\|T\|_{\UI}\},
\end{align*}
we see that \eqref{ineq:CpldSylv-UI'-1} implies \eqref{ineq:CpldSylv-UI'-2}.
\end{proof}

\Cref{lm:BdCpldSylv} has more than what we need later. In fact, we will only use \eqref{ineq:CpldSylv-2} and \eqref{ineq:CpldSylv-2'} in our later development.

The results of \Cref{lm:BdCpldSylv} have an alternative interpretation through a linear operator
\begin{equation}\label{eq:T(X):defn}
\begin{array}{cccc}
\bT\,:&\scrB:=\bbC^{s\times r}\times\bbC^{t\times r}&\,\to\,& \bT(X,Y)\in\scrB \\
  &(X,Y)&\,\to\,& (XA-BY, YA^{\HH}-B^{\HH}X),
\end{array}
\end{equation}
endowed with certain norm on $\scrB$ to make it a Banach space, including (cf. those used in \Cref{lm:BdCpldSylv})
\begin{subequations}\label{eq:norm4scrB}
\begin{equation}\label{eq:norm4scrB-1}
\|(X,Y)\|=\left\|\begin{bmatrix}
  0 & X \\
  Y & 0
\end{bmatrix}\right\|_{\UI}\equiv \left\|\begin{bmatrix}
  X & 0 \\
  0 & Y
\end{bmatrix}\right\|_{\UI}
\end{equation}
for any $(X,Y)\in\scrB$, where $\|\cdot\|_{\UI}$ is any given unitarily invariant norm. Two particular ones are
\begin{equation}\label{eq:norm4scrB-2}
\|(X,Y)\|=\max\{\|X\|_2,\|Y\|_2\}
\quad\mbox{or}\quad
\sqrt{\|X\|_{\F}^2+\|Y\|_{\F}^2},
\end{equation}
upon realizing $\|\cdot\|_{\UI}$ in \eqref{eq:norm4scrB-1} as the spectral norm or the Frobenius norm.
Another possible endowed norm on $\scrB$ is
\begin{equation}\label{eq:norm4scrB-3}
\|(X,Y)\|=\max\{\|X\|_{\UI},\|Y\|_{\UI}\}.
\end{equation}
\end{subequations}
With each endowed norm, there is an induced operator norm $\|\cdot\|$ on the linear operator from $\scrB$ to itself.
Translating the results of \Cref{lm:BdCpldSylv} yields the following corollary.

\begin{corollary}\label{cor:BdCpldSylv}
Let $A\in\bbC^{r\times r}$, $B\in\bbC^{s\times t}$, $S\in\bbC^{s\times r}$, and $T\in\bbC^{t\times r}$, and define
linear operator $\bT$
on $(\scrB:=\bbC^{s\times r}\times\bbC^{t\times r},\|\cdot\|)$ as in \eqref{eq:L(X):defn} where
$\|\cdot\|$ is given by one of those in \eqref{eq:norm4scrB}.
If $\sv(A)\cap\sv_{\ext}(B)=\emptyset$, then $\bT$ is invertible, and, furthermore, the following statements
hold.
\begin{enumerate}[{\rm (i)}]
  \item With $\|(X,Y)\|=\sqrt{\|X\|_{\F}^2+\|Y\|_{\F}^2}$ and $\delta$ as in \eqref{ineq:stew-F}, we have
        $ 
        \|\bT^{-1}\|^{-1}=\delta
        $ 
        {\rm \cite{stew:1973}};
  \item With either \eqref{eq:norm4scrB-1} or \eqref{eq:norm4scrB-3},
        if $\delta:=\sigma_{\min}(A)-\sigma_{\max}(B)>0$,
        then $\|\bT^{-1}\|^{-1}=\delta$;
  \item With \eqref{eq:norm4scrB-1} and  $\delta$ as in \eqref{ineq:stew-F}, we have
        $\|\bT^{-1}\|^{-1}\ge (2/\pi)\delta$;
  \item With \eqref{eq:norm4scrB-3} and  $\delta$ as in \eqref{ineq:stew-F}, we have
        $\|\bT^{-1}\|^{-1}\ge (1/\pi)\delta$.
\end{enumerate}
\end{corollary}

\begin{proof}
To see these, we note that $\|\bT^{-1}\|^{-1}=\min\{\gamma\,:\,\|\bT(X,Y)\|\ge\gamma\|(X,Y)\|\}$. Hence,
immediately it follows that $\|\bT^{-1}\|^{-1}\ge\delta$ for items (i) and (ii),
$\|\bT^{-1}\|^{-1}\ge (2/\pi)\delta$ for item (iii), and $\|\bT^{-1}\|^{-1}\ge (1/\pi)\delta$ for item (iv).
The equality sign in items (i) and (ii)
are achieved by letting $X=\by\bu^{\HH}$ and $Y=\bx\bv^{\HH}$, where $\bu,\,\bv, \bx,\,\by$ are unit
singular vectors of $A$ and $B$, respectively such that
$A\bv=\sigma_{\min}(A)\,\bu$ and $B\bx=\sigma_{\max}(B)\,\by$, and hence
$$
\bT(X,Y)=[\sigma_{\min}(A)-\sigma_{\max}(B)](\by\bv^{\HH},\bx\bu^{\HH}).
$$
Also because $\sv(X)$, $\sv(Y)$, $\sv(\by\bv^{\HH})$, and $\sv(\bx\bu^{\HH})$ all consist of one nonzero singular value $1$ and some copies of zeros,
\begin{gather*}
\left\|\begin{bmatrix}
  0 & X \\
  Y & 0
\end{bmatrix}\right\|_{\UI}\equiv \left\|\begin{bmatrix}
  X & 0 \\
  0 & Y
\end{bmatrix}\right\|_{\UI}
  =\left\|\begin{bmatrix}
  \by\bv^{\HH} & 0 \\
  0 & \bx\bu^{\HH}
\end{bmatrix}\right\|_{\UI}, \quad \\
\|X\|_{\UI}=\|X\|_2=1,\,\,
\|Y\|_{\UI}=\|Y\|_2=1,\\
\|\by\bv^{\HH}\|_{\UI}=\|\by\bv^{\HH}\|_2=1,\,\,
\|\bx\bu^{\HH}\|_{\UI}=\|\bx\bu^{\HH}\|_2=1,
\end{gather*}
implying $\|\bT(X,Y)\|=\delta\|(X,Y)\|$ with the respective norm on $\scrB$ as specified in item~(i) or item~(ii).
\end{proof}

The next lemma is likely known. We state it here with a proof for self-containedness.

\begin{lemma}\label{lm:SVr}
Given $B\in\bbC^{m\times n}$ and an integer $1\le r<\min\{m,n\}$, we have
\begin{align*}
\sigma_r(B)&\ge\max\left\{\sigma_{\min}(B_{(:,1:r)}),\,\sigma_{\min}(B_{(1:r,:)})\right\}\ge\sigma_{\min}(B_{(1:r,1:r)}), \\
\sigma_{r+1}(B)&\le\min\left\{\sigma_{\max}(B_{(:,r+1:n)}),\,\sigma_{\max}(B_{(r+1:m,:)})\right\}.
\end{align*}
\end{lemma}

\begin{proof}
Partition $B$ as
$$
B=\kbordermatrix{ &\sss r &\sss n-r \\
              \sss r & B_{11} &  B_{12} \\
              \sss m-r &  B_{21} & B_{22}},
$$
and then we have
$$
B^{\HH}B=\begin{bmatrix}
       B_{11}^{\HH}B_{11}+B_{21}^{\HH}B_{21} & B_{11}^{\HH}B_{12}+B_{21}^{\HH}B_{22} \\
       B_{12}^{\HH}B_{11}+B_{22}^{\HH}B_{21} & B_{12}^{\HH}B_{12}+B_{22}^{\HH}B_{22}
     \end{bmatrix}\in\bbC^{n\times n}.
$$
It follows from Fischer's minimax principle for the symmetric eigenvalue problem
(see, e.g., \cite[eq. (1.3)]{lili:2014}, \cite[p.206]{parl:1998}, \cite[p.201]{stsu:1990})
that
\begin{align}
[\sigma_r(B)]^2=\lambda_r(B^{\HH}B)
   &=\max_{\dim\cX=r}\min_{\bx\in\cX\subset\bbC^n}\frac {\bx^{\HH}(B^{\HH}B)\bx}{\bx^{\HH}\bx} \nonumber \\
   &\ge\min_{\bz\in\bbC^r}\frac {\bz^{\HH}(B_{11}^{\HH}B_{11}+B_{21}^{\HH}B_{21})\bz}{\bz^{\HH}\bz} \nonumber\\
   &\ge\min_{\bz\in\bbC^r}\frac {\bz^{\HH}(B_{11}^{\HH}B_{11})\bz}{\bz^{\HH}\bz}
   =[\sigma_{\min}(B_{11})]^2, \label{eq:minHH-r}
\end{align}
where $\cX\subset\bbC^n$ denotes a subspace of $\bbC^n$, $\lambda_r(B^{\HH}B)$ is the $r$th largest eigenvalue of $B^{\HH}B$, and  the first inequality is due to limiting
the subspace to the one composed of vectors with their last $n-r$ entries being to $0$.
This proves $\sigma_r(B)\ge\sigma_{\min}(B_{(:,1:r)})\ge\sigma_{\min}(B_{(1:r,1:r)})$.
Next, using $[\sigma_r(B)]^2=\lambda_r(BB^{\HH})$, in the same way as in \eqref{eq:minHH-r}, we can get
$\sigma_r(B)\ge\sigma_{\min}(B_{(1:r,:)})\ge\sigma_{\min}(B_{(1:r,1:r)})$. This completes the proof of the inequalities
for $\sigma_r(B)$.

Analogously, again by Fischer's minimax principle, we have
\begin{align}
[\sigma_{r+1}(B)]^2=\lambda_{r+1}(B^{\HH}B)
   &=\min_{\dim\cX=n-r}\max_{\bx\in\cX\subset\bbC^n}\frac {\bx^{\HH}(B^{\HH}B)\bx}{\bx^{\HH}\bx} \nonumber\\
   &\le\max_{\bz\in\bbC^{n-r}}\frac {\bz^{\HH}(B_{12}^{\HH}B_{12}+B_{22}^{\HH}B_{22})\bz}{\bz^{\HH}\bz} \nonumber\\
   &=\lambda_{\max}(B_{12}^{\HH}B_{12}+B_{22}^{\HH}B_{22})
    =[\sigma_{\max}(B_{(:,r+1:n)})]^2, \label{eq:maxHH-r}
\end{align}
where $\lambda_{r+1}(B^{\HH}B)$ is the $(r+1)$st largest eigenvalue of $B^{\HH}B$, and the first inequality is due to limiting
the subspace to the one composed of vectors with their first $r$ entries being to $0$.
Next, using $[\sigma_{r+1}(B)]^2=\lambda_{r+1}(BB^{\HH})$, in the same way as in \eqref{eq:maxHH-r}, we can get
$\sigma_{r+1}(B)\le\sigma_{\max}(B_{(r+1:m,:)})$.
\end{proof}

\section{Main Result}\label{sec:main}
In order to achieve \eqref{eq:SVD4tG-almost:apx}, we need
some $\Omega$ and $\Gamma$ to satisfy
\begin{subequations}\label{eq:GtG:syl:apx}
\begin{align}
\Gamma (G_1+E_{11})-(G_2+E_{22})\Omega &= E_{21}-\Gamma\, E_{12}\Omega,  \label{eq:GtG:syl-1:apx}\\
\Omega (G_1+E_{11})^{\HH}-(G_2+E_{22})^{\HH}\Gamma &=E_{12}^{\HH}-\Omega E_{21}^{\HH} \Gamma, \label{eq:GtG:syl-2:apx}
\end{align}
\end{subequations}
obtained from setting off-diagonal blocks of $\check U^{\HH}\wtd G \check V$, partitioned accordingly, to $0$.

In what follows, we will use Lemma~\ref{lm:stew-lm} to prove the existence of a solution pair
$(\Gamma,\Omega)\in\bbC^{(m-r)\times r}\times\bbC^{(n-r)\times r}$ to \eqref{eq:GtG:syl:apx}
with an upper bound under certain conditions.

Keeping in mind that our goal is to create a variant
of \Cref{thm:stew:1973SVD} in a unitarily invariant norm and the spectral norm, and hopefully the variant for the special case of
the Frobenius norm is better than \Cref{thm:stew:1973SVD} in terms of both weaker conditions and stronger results.
In the setting of \Cref{lm:stew-lm}, we will use the Banach space
$$
\scrB:=\bbC^{(m-r)\times r}\times\bbC^{(n-r)\times r}
$$
endowed with one of the two norms: for $(\Gamma,\Omega)\in\scrB$,
\begin{subequations}\label{eq:scrB-norm}
\begin{align}
\|(\Gamma,\Omega)\|
    &:=\left\|\begin{bmatrix}
               0 & \Gamma \\
               \Omega & 0
             \end{bmatrix}\right\|_{\UI}\equiv\left\|\begin{bmatrix}
               \Gamma & 0\\
               0 & \Omega
             \end{bmatrix}\right\|_{\UI}, \label{eq:scrB-norm-1} \\
\|(\Gamma,\Omega)\|&:=\max\{\|\Gamma\|_{\UI},\|\Omega\|_{\UI}\}. \label{eq:scrB-norm-2}
\end{align}
\end{subequations}
The two endowed norms become one for the case $\|\cdot\|_{\UI}=\|\cdot\|_2$, the spectral norm, but otherwise are different.
The linear operator $\bT\,:\,\scrB\to\scrB$ is given by
\begin{equation}\label{eq:bT-map}
\bT(\Gamma,\Omega)=\Big(\Gamma (G_1+E_{11})-(G_2+E_{22})\Omega,
                         \Omega (G_1+E_{11})^{\HH}-(G_2+E_{22})^{\HH}\Gamma\Big),
\end{equation}
and the continuous function $\bphi\,:\,\scrB\to\scrB$ is
\begin{equation}\label{eq:bphi-map}
\bphi((\Gamma,\Omega))=\left(\Gamma\, E_{12}\Omega,\Omega E_{21}^{\HH} \Gamma\right).
\end{equation}
It is not hard to verify that $\|(\Gamma,\Omega)\|$ defined in \eqref{eq:scrB-norm} is indeed a norm on $\scrB$.

Compactly, the two equations in \eqref{eq:GtG:syl:apx} can be merged into one to take the form
\begin{equation}\label{eq:GtG:syl':apx}
\bT(\Gamma,\Omega)=(E_{21},E_{12}^{\HH})-\bphi((\Gamma,\Omega)).
\end{equation}

It remains to verify the conditions of Lemma~\ref{lm:stew-lm} for $\bT$ and $\bphi$ we just defined.
This is done in the next two lemmas.
Let
\begin{subequations}\label{eq:quantities-apx}
\begin{align}
\delta&=\min_{\mu\in\sv(G_1),\,\nu\in\sv_{\ext}(G_2)}\,|\mu-\nu|, \label{eq:quantities-apx-2} \\
\munderbar\delta&=\delta-\|E_{11}\|_2-\|E_{22}\|_2,  \label{eq:quantities-apx-3} \\
\varepsilon&=\max\{\|E_{12}\|_2,\|E_{21}\|_2\}. \label{eq:quantities-apx-4}
\end{align}
\end{subequations}

\begin{lemma}\label{lm:veri-stew-bT}
Suppose $\munderbar\delta>0$. Dependent of different cases, we have for any $(\Gamma,\Omega)\in\scrB$
\begin{equation}\label{eq:bT-bd}
\|\bT(\Gamma,\Omega)\|\ge\frac {\munderbar\delta}c\,\|(\Gamma,\Omega)\|,
\end{equation}
which implies $\|\bT^{-1}\|^{-1}\ge\munderbar\delta/c$, where $\munderbar\delta$ is as in \eqref{eq:quantities-apx-3}, and
\begin{enumerate}[{\rm (i)}]
  \item with the endowed norm in \eqref{eq:scrB-norm-1},
        \begin{equation}\label{eq:constant-c}
        c=
        \begin{cases}
          1, \quad&\mbox{if $\sigma_{\min}(G_1)>\sigma_{\max}(G_2)$ or $\|\cdot\|_{\UI}=\|\cdot\|_{\F}$}, \\
          \pi/2, \quad&\mbox{otherwise};
        \end{cases}
        \end{equation}
  \item with the endowed norm in \eqref{eq:scrB-norm-2},
        \begin{equation}\label{eq:constant-c'}
        c=
        \begin{cases}
          1, \quad&\mbox{if $\sigma_{\min}(G_1)>\sigma_{\max}(G_2)$}, \\
          \pi, \quad&\mbox{otherwise}.
        \end{cases}
        \end{equation}
\end{enumerate}
\end{lemma}

\begin{proof}
Notice that $\bT(\Gamma,\Omega)=(S,T)$ consists of a set of two coupled Sylvester equations in \eqref{eq:CpldSyvl}
with
$$
A=G_1+E_{11},\,\,
B=G_2+E_{22},
$$
and $\Gamma$ and $\Omega$ correspond to $X$ and $Y$ there, respectively. We claim that
\begin{equation}\label{eq:veri-stew-bT:pf-1}
\munderbar\delta\le\min_{\mu\in\sv(G_1+E_{11}),\,\nu\in\sv_{\ext}(G_2+E_{22})}\,|\mu-\nu|=|\mu'-\nu'|.
\end{equation}
where $\mu'\in\sv(G_1+E_{11}),\,\nu'\in\sv_{\ext}(G_2+E_{22})$ achieve the minimum.
By Mirsky's theorem \cite[p.204]{stsu:1990}, there are $\mu\in\sv(G_1),\,\nu\in\sv_{\ext}(G_2)$ such that
$|\mu'-\mu|\le \|E_{11}\|_2$ and $|\nu'-\nu|\le \|E_{22}\|_2$. Hence
\begin{align*}
|\mu'-\nu'|
   &=|\mu-\nu+(\mu'-\mu)-(\nu'-\nu)| \\
   &\ge|\mu-\nu|-|\mu'-\mu|-|\nu'-\nu| \\
   &\ge\delta-\|E_{11}\|_2-\|E_{22}\|_2=\munderbar\delta,
\end{align*}
yielding \eqref{eq:veri-stew-bT:pf-1}.

If $\sigma_{\min}(G_1)>\sigma_{\max}(G_2)$, then
$\delta=\sigma_{\min}(G_1)-\sigma_{\max}(G_2)>0$, and hence
\begin{align*}
0<\munderbar\delta&=\sigma_{\min}(G_1)-\sigma_{\max}(G_2)-\|E_{11}\|_2-\|E_{22}\|_2 \\
    &=[\sigma_{\min}(G_1)-\|E_{11}\|_2]-[\sigma_{\max}(G_2)+\|E_{22}\|_2].
\end{align*}
Keep in mind that
$$
\sigma_{\min}(G_1+E_{11})\ge\sigma_{\min}(G_1)-\|E_{11}\|_2, \quad
\sigma_{\max}(G_2+E_{22})\le\sigma_{\max}(G_2)+\|E_{22}\|_2
$$
by Mirsky's theorem \cite[p.204]{stsu:1990}, and hence
$$
\sv(A)\subset [\sigma_{\min}(G_1)-\|E_{11}\|_2,\infty), \quad
\sv_{\ext}(B)\subset [0,\sigma_{\max}(G_2)+\|E_{22}\|_2].
$$
This lemma is a consequence of \Cref{lm:BdCpldSylv}.
\end{proof}

\begin{remark}\label{rk:veri-stew-bT}
There are a few comments in order, regarding the assumptions in \Cref{lm:veri-stew-bT} that ensure \eqref{eq:bT-bd}.
\begin{enumerate}[1)]
  \item Always $\sv(G_1)\cap\sv_{\ext}(G_2)=\sv(G_1+E_{11})\cap\sv_{\ext}(G_2+E_{22})=\emptyset$, guaranteed by $\munderbar\delta>0$;
  \item If $m\ne n$, then $0$ is an element of both $\sv_{\ext}(G_2)$ and $\sv_{\ext}(G_2+E_{22})$ and hence
        $\sigma_{\min}(G_1)>0$ and $\sigma_{\min}(G_1+E_{11})>0$;
  \item Two different ways of separation between $\sv(G_1)$ and $\sv_{\ext}(G_2)$ are assumed:
        \begin{enumerate}[(i)]
          \item simply         $\sv(G_1)\cap\sv_{\ext}(G_2)=\emptyset$;
          \item $\sigma_{\min}(G_1)>\sigma_{\max}(G_2)$, i..e,
                the interval $[\sigma_{\max}(G_2),\sigma_{\min}(G_1)]$ separates $\sv(G_1)$ from $\sv_{\ext}(G_2)$.
        \end{enumerate}
        Assumption (i) of separation is weaker than Assumption (ii) of separation.
        Each implies the same way of separation between $\sv(G_1+E_{11})$ and $\sv_{\ext}(G_2+E_{22})$ by $\munderbar\delta>0$;
  \item The endowed norm \eqref{eq:scrB-norm-1} works with both assumptions of separation on the singular values.
        Correspondingly, $c=1$ always for the Frobenius norm, and $c=1$ under Assumption (ii) of separation
        above and $\pi/2$ otherwise;
  \item The endowed norm \eqref{eq:scrB-norm-2} works with both assumptions of separation, too.
        Correspondingly, $c=1$ always for the Frobenius norm, and $c=1$ under Assumption (ii) of separation
        and $\pi$ otherwise.
\end{enumerate}
Finally, because of \eqref{eq:UI=2:equiv}, items (i) and (ii) of \Cref{lm:veri-stew-bT}
        overlap at the case $\|\cdot\|_{\UI}=\|\cdot\|_2$ and $\sigma_{\min}(G_1)>\sigma_{\max}(G_2)$.
\end{remark}

In what follows, our representation will assume that one of the endowed norm $\|(\cdot,\cdot)\|$
in \eqref{eq:scrB-norm} is selected and fixed, and,  along with it, the part of \Cref{lm:veri-stew-bT}, unless explicitly stated otherwise.

\begin{lemma}\label{lm:veri-stew-bphi}
For the continuous function $\bphi$  defined in \eqref{eq:bphi-map},
we have
\begin{align*}
  \mbox{\rm (i)}  &\quad \|\bphi((\Gamma,\Omega))\|\le {\varepsilon}\,\|(\Gamma,\Omega)\|^2, \\
  \mbox{\rm (ii)} &\quad \|\bphi((\Gamma,\Omega))-\bphi((\what\Gamma,\what\Omega))\|
                    \le 2\varepsilon\max\{\|(\Gamma,\Omega)\|,\|(\what\Gamma,\what\Omega)\|\}\|(\Gamma-\what\Gamma,\Omega-\what\Omega)\|,
\end{align*}
where $\varepsilon=\max\{\|E_{12}\|_2,\|E_{21}\|_2\}$ is as in \eqref{eq:quantities-apx-4}.
\end{lemma}

\begin{proof}
With \eqref{eq:scrB-norm-1}, we have
\begin{align}
\|\bphi((\Gamma,\Omega))\|
  &=\left\|\begin{bmatrix}
             \Gamma\,E_{12}\Omega & 0 \\
             0 & \Omega E_{21}^{\HH} \Gamma
           \end{bmatrix}\right\|_{\UI}
   =\left\|\begin{bmatrix}
             \Gamma & 0 \\
             0 & \Omega
           \end{bmatrix}
   \begin{bmatrix}
             E_{12} & 0 \\
             0 &  E_{21}^{\HH}
           \end{bmatrix}\begin{bmatrix}
             \Omega & 0 \\
             0 &  \Gamma
           \end{bmatrix}\right\|_{\UI} \nonumber \\
  &\le\left\|\begin{bmatrix}
             \Gamma & 0 \\
             0 & \Omega
           \end{bmatrix}\right\|_{\UI}
   \left\|\begin{bmatrix}
             E_{12} & 0 \\
             0 &  E_{21}^{\HH}
           \end{bmatrix}\right\|_2\left\|\begin{bmatrix}
             \Omega & 0 \\
             0 &  \Gamma
           \end{bmatrix}\right\|_{\UI} \nonumber \\
  &=\max\{\|E_{12}\|_2,\|E_{21}\|_2\}\times\|(\Gamma,\Omega)\|^2; \label{eq:veri-stew-bphi:pf-1}
\end{align}
with \eqref{eq:scrB-norm-2}, we have
\begin{align}
\|\bphi((\Gamma,\Omega))\|
  &=\max\left\{\|\Gamma\,E_{12}\Omega\|_{\UI},\|\Omega E_{21}^{\HH} \Gamma\|_{\UI}\right\} \nonumber\\
  &\le\max\{\|E_{12}\|_2,\|E_{21}\|_2\}\times\|\Gamma\|_{\UI}\|\Omega\|_{\UI} \nonumber\\
  &\le\max\{\|E_{12}\|_2,\|E_{21}\|_2\}\times\|(\Gamma,\Omega)\|^2. \label{eq:veri-stew-bphi:pf-2}
\end{align}
This proves the inequality in item (i). For item (ii), we note
$$
\bphi((\Gamma,\Omega))-\bphi((\what\Gamma,\what\Omega))
  =\left(\Gamma\,E_{12}\Omega-\what\Gamma\,E_{12}\what\Omega,\Omega E_{21}^{\HH} \Gamma-\what\Omega E_{21}^{\HH}\what \Gamma\right).
$$
For each of the two components, we have
\begin{subequations}\label{eq:veri-stew-bphi:pf-3}
\begin{align}
\Gamma\,E_{12}\Omega-\what\Gamma\,E_{12}\what\Omega
  &=(\Gamma-\what\Gamma)\,E_{12}\Omega+\what\Gamma\,E_{12}(\Omega-\what\Omega), \label{eq:veri-stew-bphi:pf-3a} \\
\Omega E_{21}^{\HH} \Gamma-\what\Omega E_{21}^{\HH}\what \Gamma
  &=(\Omega-\what\Omega)\, E_{21}^{\HH} \Gamma+\what\Omega E_{21}^{\HH}(\Gamma-\what \Gamma). \label{eq:veri-stew-bphi:pf-3b}
\end{align}
\end{subequations}
With \eqref{eq:scrB-norm-1}, we get, similar to the derivation in \eqref{eq:veri-stew-bphi:pf-1},
\begin{align*}
\|\bphi((\Gamma,\Omega))&-\bphi((\what\Gamma,\what\Omega))\| \\
  &=\left\|\begin{bmatrix}
             (\Gamma-\what\Gamma)\,E_{12}\Omega+\what\Gamma\,E_{12}(\Omega-\what\Omega) & 0 \\
             0 & (\Omega-\what\Omega)\, E_{21}^{\HH} \Gamma+\what\Omega E_{21}^{\HH}(\Gamma-\what \Gamma)
           \end{bmatrix}\right\|_{\UI} \\
  &\le\left\|\begin{bmatrix}
             (\Gamma-\what\Gamma)\,E_{12}\Omega & 0 \\
             0 & (\Omega-\what\Omega)\, E_{21}^{\HH} \Gamma
           \end{bmatrix}\right\|_{\UI}
       +\left\|\begin{bmatrix}
             \what\Gamma\,E_{12}(\Omega-\what\Omega) & 0 \\
             0 & \what\Omega E_{21}^{\HH}(\Gamma-\what \Gamma)
           \end{bmatrix}\right\|_{\UI} \\
  &\le\varepsilon\,\|(\Gamma-\what\Gamma,\Omega-\what\Omega)\|
      \big(\|(\Gamma,\Omega)\|+\|(\what\Gamma,\what\Omega)\|\big) \\
  &\le 2\varepsilon\max\{\|(\Gamma,\Omega)\|,\|(\what\Gamma,\what\Omega)\|\}\|(\Gamma-\what\Gamma,\Omega-\what\Omega)\|;
\end{align*}
with \eqref{eq:scrB-norm-2}, we get, similar to the derivation in \eqref{eq:veri-stew-bphi:pf-2},
\begin{align*}
&\max\{\|\Gamma\,E_{12}\Omega-\what\Gamma\,E_{12}\what\Omega\|_{\UI},
       \|\Omega E_{21}^{\HH} \Gamma-\what\Omega E_{21}^{\HH}\what \Gamma\|_{\UI}\} \\
  &\qquad\le\varepsilon\max\{\|\Gamma-\what\Gamma\|_{\UI}\|\Omega\|_{\UI}+\|\what\Gamma\|_{\UI}\|\Omega-\what\Omega\|_{\UI},
         \|\Omega-\what\Omega\|_{\UI}\|\Gamma\|_{\UI}+\|\what\Omega\|_{\UI} \|\Gamma-\what\Gamma\|_{\UI}\} \\
  &\qquad\le\varepsilon\max\{\|\Gamma-\what\Gamma\|_{\UI},\|\Omega-\what\Omega\|_{\UI}\}\cdot
                       \max\{\|\Omega\|_{\UI}+\|\what\Gamma\|_{\UI},
                                \|\Gamma\|_{\UI}+\|\what\Omega\|_{\UI}\} \\
  &\qquad\le\varepsilon\max\{\|\Gamma-\what\Gamma\|_{\UI},\|\Omega-\what\Omega\|_{\UI}\}\cdot
                       2\max\{\|(\Gamma,\Omega)\|,\|(\what\Gamma,\what\Omega)\|\} \\
  &\qquad= 2\varepsilon\,\max\{\|(\Gamma,\Omega)\|,\|(\what\Gamma,\what\Omega)\|\}\|(\Gamma-\what\Gamma,\Omega-\what\Omega)\|,
\end{align*}
completing the proof of item (ii).
\end{proof}

Next we apply Lemma~\ref{lm:stew-lm} to ensure a particular solution $(\Gamma,\Omega)\in\scrB$ to
\eqref{eq:GtG:syl':apx}.

\begin{lemma}\label{lm:eq:GtG:syl:apx}
Let $\delta$, $\munderbar\delta$,  and $\varepsilon$ be defined as in \eqref{eq:quantities-apx}, and suppose
the conditions in \Cref{lm:veri-stew-bT} that ensure \eqref{eq:bT-bd} with constant $c$ as
specified there.
If
\begin{equation}\label{eq:tG:small-pert}
\munderbar\delta>0
\quad\mbox{and}\quad
\kappa_2:=\frac {c^2\varepsilon}{\munderbar\delta^2}\,\|(E_{21},E_{12}^{\HH})\|<\frac 14,
\end{equation}
then \eqref{eq:GtG:syl':apx} has a solution $(\Gamma,\Omega)$ that satisfies
\begin{equation}\label{eq:bd4SVDtG}
\|(\Gamma,\Omega)\|
 \le\frac {1+\sqrt{1-4\kappa_2}}{1-2\kappa_2+\sqrt{1-4\kappa_2}}
         \frac {c\|(E_{21},E_{12}^{\HH})\|}{\munderbar\delta}<2\,\frac {c\|(E_{21},E_{12}^{\HH})\|}{\munderbar\delta}.
\end{equation}
Here $\|(E_{21},E_{12}^{\HH})\|$ and $\|(\Gamma,\Omega)\|$ are understood as the same one of the endowed norms in \eqref{eq:scrB-norm} under consideration.
\end{lemma}

\begin{proof}
In light of \Cref{lm:veri-stew-bT,lm:veri-stew-bphi}, we find the conclusion is a straightforward consequence of Lemma~\ref{lm:stew-lm}
with $\hat\varepsilon=\varepsilon$, $\hat\delta=\munderbar\delta/c$, and $\|\bg\|=\|(E_{21},E_{12}^{\HH})\|$.
\end{proof}

Finally, we state the main result of this section.

\begin{theorem}\label{thm:SVD-almost}
Given $G,\,\wtd G\in\bbC^{m\times n}$, let $G$ be decomposed as in \eqref{eq:G:SVD:apx}, and partition
$U^{\HH}\wtd GV$ according to \eqref{eq:tG:apx-1}. Let  $\delta$,
$\munderbar\delta$,  and $\varepsilon$ be defined as in \eqref{eq:quantities-apx}. If \eqref{eq:tG:small-pert} is satisfied, then the following statements hold:
\begin{enumerate}[{\rm (a)}]
  \item there exists a solution
        $(\Gamma,\Omega)\in\bbC^{(m-r)\times r}\times\bbC^{(n-r)\times r}$ to \eqref{eq:GtG:syl:apx},
        satisfying \eqref{eq:bd4SVDtG};
  \item $\wtd G$ admits the decomposition \eqref{eq:SVD4tG-almost:apx},
        and the singular values of $\wtd G$ is the multiset union of those of
        \begin{subequations}\label{eq:checkSigma1}
        \begin{align}
        \check G_1&=\check U_1^{\HH}\wtd G\check V_1 \nonumber \\
          &=(I+\Gamma^{\HH}\Gamma)^{1/2}(G_1+E_{11} +  E_{12}\Omega)(I+\Omega^{\HH}\Omega)^{-1/2}
                 \label{eq:checkSigma1-1}\\
          &=(I+\Gamma^{\HH}\Gamma)^{-1/2}(G_1+E_{11} + \Gamma^{\HH} E_{21})(I+\Omega^{\HH}\Omega)^{1/2},
                 \label{eq:checkSigma1-2}
        \end{align}
        \end{subequations}
        and
        \begin{subequations}\label{eq:checkSigma2}
        \begin{align}
        \check G_2&=\check U_2^{\HH}\wtd G\check V_2 \nonumber \\
          &=(I+\Gamma\Gamma^{\HH})^{1/2}(G_2+E_{22} -  E_{21}\Omega^{\HH})(I+\Omega\Omega^{\HH})^{-1/2}
                \label{eq:checkSigma2-1}\\
          &=(I+\Gamma\Gamma^{\HH})^{-1/2}(G_2+E_{22} - \Gamma E_{12})(I+\Omega\Omega^{\HH})^{1/2};
                \label{eq:checkSigma2-2}
        \end{align}
        \end{subequations}
  \item We have
        \begin{subequations}\label{eq:checkSigma1Sigma2}
        \begin{align}
        \sigma_{\min}(\check G_1)\ge\sigma_{\min}(G_1)-\|E_{11}\|_2
                       -2c\,\frac {\varepsilon\,\|(E_{21},E_{12}^{\HH})\|}{\munderbar\delta}, \label{eq:checkSigma1Sigma2-1} \\
        \sigma_{\max}(\check G_2)\le\sigma_{\max}(G_2)+\|E_{22}\|_2+2c\,\frac {\varepsilon\,\|(E_{21},E_{12}^{\HH})\|}{\munderbar\delta},
                            \label{eq:checkSigma1Sigma2-2}
        \end{align}
        \end{subequations}
        where $\sigma_{\min}(\check G_1)$ and $\sigma_{\max}(\check G_2)$
        are the smallest singular value of $\check G_1$ and the largest singular value of $\check G_2$,
        respectively;
  \item The left and right singular subspaces of $\wtd G$ associated with the part of its singular values $\sv(\check G)$
        are spanned by the columns of
        \begin{subequations}\nonumber
        \begin{align}
        \check U_1 &=(U_1+U_2\Gamma)(I+\Gamma^{\HH}\Gamma)^{-1/2}, \label{eq:checkU1V1-1:apx}\\
        \check V_1 &=(V_1+V_2\Omega)(I+\Omega^{\HH}\Omega)^{-1/2}, \label{eq:checkU1V1-2:apx}
        \end{align}
        \end{subequations}
        respectively. In particular,
        \begin{subequations}\label{eq:diff(U1V1):approxBT}
        \begin{align}
        \|\check U_1-U_1\|_{\UI}
           &\le\|\Gamma\|_{\UI}
            \le\|(\Gamma,\Omega)\|
            \le2\,\frac {c\|(E_{21},E_{12}^{\HH})\|}{\munderbar\delta}, \label{eq:diff(U1V1)-1:approxBT}\\
        \|\check V_1-V_1\|_{\UI}
           &\le\|\Omega\|_{\UI}
            \le\|(\Gamma,\Omega)\|
            \le2\,\frac {c\|(E_{21},E_{12}^{\HH})\|}{\munderbar\delta}. \label{eq:diff(U1V1)-2:approxBT}
        \end{align}
        \end{subequations}
\end{enumerate}
\end{theorem}

\begin{proof}
Only item (c) and the inequalities in \eqref{eq:diff(U1V1):approxBT} of item (d) need proofs. For item (c),
using \eqref{eq:checkSigma1-1} for $\check G_1$ and \eqref{eq:checkSigma1-2} for $\check G_1$ in $\check G_1^{\HH}$
below,
we get
$$
\check G_1\check G_1^{\HH}
  =(I+\Gamma^{\HH}\Gamma)^{1/2}(G_1+E_{11} +  E_{12}\Omega)
               (G_1+E_{11} + \Gamma^{\HH} E_{21})^{\HH}(I+\Gamma^{\HH}\Gamma)^{-1/2}.
$$
Hence
$$
\big[\sigma_{\min}(\check G_1)\big]^2=\lambda_{\min}(\check G_1\check G_1^{\HH})
 =\lambda_{\min}\left((G_1+E_{11} +  E_{12}\Omega)(G_1+E_{11} + \Gamma^{\HH} E_{21})^{\HH}\right),
$$
where $\lambda_{\min}(\cdot)$ is the smallest eigenvalues of a Hermitian matrix.
Therefore, with the help of \eqref{eq:bd4SVDtG}, we have
\begin{align}
\big[\sigma_{\min}(\check G_1)\big]^2
  &=\left\|\left[(G_1+E_{11} +  E_{12}\Omega)(G_1+E_{11} + \Gamma^{\HH} E_{21})^{\HH}\right]^{-1}\right\|_2^{-1}
           \nonumber\\
  &\ge\left\|(G_1+E_{11} + \Gamma^{\HH} E_{21})^{-\HH}\right\|_2^{-1}
      \left\|(G_1+E_{11} +  E_{12}\Omega)^{-1}\right\|_2^{-1} \nonumber\\
  &\ge \left(\sigma_{\min}(G_1)-\|E_{11}\|_2-\|\Gamma\|_2\|E_{21}\|_2\right)
       \left(\sigma_{\min}(G_1)-\|E_{11}\|_2-\|E_{12}\|_2\|\Omega\|_2\right) \nonumber\\
  &\ge\Big(\sigma_{\min}(G_1)-\|E_{11}\|_2-\|(\Gamma,\Omega)\|\varepsilon\Big)^2 \nonumber\\
  &\ge\left(\sigma_{\min}(G_1)-\|E_{11}\|_2
            -2c\,\frac {\varepsilon\,\|(E_{21},E_{12}^{\HH})\|}{\munderbar\delta}\right)^2,\label{eq:SVD-almost:pf-1}
\end{align}
yielding the first inequality in \eqref{eq:checkSigma1Sigma2}. Similarly, we get the second inequality there by
using \eqref{eq:checkSigma2}.

Now we prove the inequalities in \eqref{eq:diff(U1V1):approxBT}. We have
\begin{align*}
\check U_1-U_1&=U_1\big[(I+\Gamma^{\HH}\Gamma)^{-1/2}-I\big]+U_2\Gamma(I+\Gamma^{\HH}\Gamma)^{-1/2} \\
   &=[-U_1,U_2]\begin{bmatrix}
              I-(I+\Gamma^{\HH}\Gamma)^{-1/2} \\
              \Gamma(I+\Gamma^{\HH}\Gamma)^{-1/2}
            \end{bmatrix}.
\end{align*}
Let $\Gamma=Z\Xi W^{\HH}$ be the SVD of $\Gamma$. We find
$$
\begin{bmatrix}
              I-(I+\Gamma^{\HH}\Gamma)^{-1/2} \\
              \Gamma(I+\Gamma^{\HH}\Gamma)^{-1/2}
            \end{bmatrix}=\begin{bmatrix}
                W &  \\
                 & Z
              \end{bmatrix}
            \begin{bmatrix}
              I-(I+\Xi^{\HH}\Xi)^{-1/2} \\
              \Xi(I+\Xi^{\HH}\Xi)^{-1/2}
            \end{bmatrix} W^{\HH},
$$
where for the middle matrix on the right, $I-(I+\Xi^{\HH}\Xi)^{-1/2}$ is diagonal and
$\Xi(I+\Xi^{\HH}\Xi)^{-1/2}$ is leading diagonal. Hence the  singular values of the middle matrix are given by: for each singular value $\gamma$ of $\Gamma$,
\begin{align}
\sqrt{\left(1-\frac 1{\sqrt{1+\gamma^2}}\right)^2+\left(\frac {\gamma}{\sqrt{1+\gamma^2}}\right)^2}
  &=\sqrt{2\left(1-\frac 1{\sqrt{1+\gamma^2}}\right)} \label{eq:SVD-almost:pf-2}\\
  &=\frac {\sqrt 2\,\gamma}{\big[\sqrt{1+\gamma^2}(\sqrt{1+\gamma^2}+1)\big]^{1/2}} \label{eq:SVD-almost:pf-3} \\
  &\le\gamma. \nonumber
\end{align}
Therefore, we
get\footnote {By \eqref{eq:SVD-almost:pf-2} and \eqref{eq:SVD-almost:pf-3}, we conclude that for the spectral norm
   $$
   \|\check U_1-U_1\|_2
           =\frac {\sqrt 2\,\|\Gamma\|_2}{\big[\sqrt{1+\|\Gamma\|_2^2}(\sqrt{1+\|\Gamma\|_2^2}+1)\big]^{1/2}}, \,\,
        \|\check V_1-V_1\|_2
           =\frac {\sqrt 2\,\|\Omega\|_2}{\big[\sqrt{1+\|\Omega\|_2^2}(\sqrt{1+\|\Omega\|_2^2}+1)\big]^{1/2}}.
   $$
   }
$$
\|\check U_1-U_1\|_{\UI}
            =\left\| \begin{bmatrix}
              I-(I+\Xi^{\HH}\Xi)^{-1/2} \\
              \Xi(I+\Xi^{\HH}\Xi)^{-1/2}
            \end{bmatrix}\right\|_{\UI}
            \le\|\Gamma\|_{\UI},
$$
yielding \eqref{eq:diff(U1V1)-1:approxBT} in light of \eqref{eq:bd4SVDtG}.
Similarly, we have \eqref{eq:diff(U1V1)-2:approxBT}.
\end{proof}

The lower and upper bound on $\sigma_{\min}(\check G_1)$ and $\sigma_{\max}(\check G_2)$, respectively, in
\eqref{eq:checkSigma1Sigma2}, although always true, do not provide useful information, unless also
$\sigma_{\min}(G_1)>\sigma_{\max}(G_2)$, in which case it can be used to establish a sufficient condition
to ensure $\sigma_{\min}(\check G_1)>\sigma_{\max}(\check G_2)$.

\begin{corollary}\label{cor:SVD-almost}
Given $G,\,\wtd G\in\bbC^{m\times n}$, let $G$ be decomposed as in \eqref{eq:G:SVD:apx}, and partition
$U^{\HH}\wtd GV$ according to \eqref{eq:tG:apx-1}. Let  $\delta$,
$\munderbar\delta$,  and $\varepsilon$ be defined as in \eqref{eq:quantities-apx}, and suppose
$\sigma_{\min}(G_1)>\sigma_{\max}(G_2)$. If \eqref{eq:tG:small-pert} with $c=1$ is satisfied, then
the top $r$ singular values of $\wtd G$
are exactly the $r$ singular values of $\check G_1$, and
\begin{equation}\label{eq:checkSigma1-improved}
\sigma_{\min}(G_1)-\|E_{11}\|_2
  \le\sigma_{\min}(\check G_1)
  \le\sigma_{\min}(G_1)+\|E_{11}\|_2+\frac {2\varepsilon^2}
                                      {\munderbar\delta+\sqrt{\munderbar\delta^2+4\varepsilon^2}}.
\end{equation}
\end{corollary}

\begin{proof}
We have all conclusions of \Cref{thm:SVD-almost}.
By \eqref{eq:checkSigma1Sigma2}, we have
$$
\sigma_{\min}(\check G_1)-\sigma_{\max}(\check G_2)
  \ge \munderbar\delta-4 \frac {\varepsilon\,\|(E_{21},E_{12}^{\HH})\|}{\munderbar\delta}
  =\munderbar\delta\left[1-4\frac {\varepsilon\,\|(E_{21},E_{12}^{\HH})\|}{\munderbar\delta^2}\right]
  >0,
$$
which says that the top $r$ singular values of $\wtd G$
        are exactly the $r$ singular values of $\check G_1$. As a result, applying \Cref{lm:SVr} to \eqref{eq:tG:apx-1},
        we have
\begin{equation}\label{eq:SVD-almost:pf-1(cor)}
\sigma_{\min}(\check G_1)=\sigma_r\big( U^{\HH}\wtd GV\big)
   \ge\sigma_{\min}(G_1+E_{11})
   \ge\sigma_{\min}(G_1)-\|E_{11}\|_2,
\end{equation}
where the last inequality in \eqref{eq:SVD-almost:pf-1(cor)} is a consequence of Mirsky's theorem \cite[p.204]{stsu:1990}.
This proves the first inequality in \eqref{eq:checkSigma1-improved}.
It can be seen that
$$
\sigma_{\min}(G_1+E_{11})-\sigma_{\max}(G_2+E_{22})
  \ge\sigma_{\min}(G_1)-\|E_{11}\|_2-\big[\sigma_{\max}(G_2)+\|E_{22}\|_2\big]
  =\munderbar\delta,
$$
and hence by \cite[Theorem 3]{lili:2005} combining with \eqref{eq:SVD-almost:pf-1(cor)}, we get
\begin{equation}\label{eq:SVD-almost:pf-2(cor)}
\sigma_{\min}(\check G_1)-\sigma_{\min}(G_1+E_{11})
    \le\frac {2\varepsilon^2}{\munderbar\delta+\sqrt{\munderbar\delta^2+4\varepsilon^2}},
\end{equation}
yielding the second inequality in \eqref{eq:checkSigma1-improved}.
\end{proof}

The sharper lower bound on $\sigma_{\min}(\check G_1)$ in \eqref{eq:checkSigma1-improved} than the one by
        the first inequality in \eqref{eq:checkSigma1Sigma2} is only made possible by first establishing that
        the top $r$ singular values of $\wtd G$ are exactly the $r$ singular values of $\check G_1$, which
        relies on the first inequality in \eqref{eq:checkSigma1Sigma2} for a proof, however.
More than \eqref{eq:SVD-almost:pf-2(cor)} which is just for the smallest singular value only, \cite[Theorem 3]{lili:2005} yields
$$
|\sigma_i(\check G_1)-\sigma_i(G_1+E_{11})|\le \frac {2\varepsilon^2}{\munderbar\delta+\sqrt{\munderbar\delta^2+4\varepsilon^2}}
\quad\mbox{for $1\le i\le r$}
$$
under the conditions of \Cref{cor:SVD-almost}.

\section{Discussions}\label{sec:disscuss}
\Cref{thm:SVD-almost} serves the same purpose as Stewart's \cite[Theorem~6.4]{stew:1973}, but the former
provides a variety of choices of norms for the Banach space $\scrB$ and, as a consequence, a number of results dependent of
circumstances to use.

Let us begin by realizing \Cref{thm:SVD-almost} and \Cref{cor:SVD-almost} with unitarily invariant norm $\|\cdot\|_{\UI}$
being set to the Frobenius norm, in order to compare with \Cref{thm:stew:1973SVD} \cite[Theorem~6.4]{stew:1973}.
We have two endowed norms from \eqref{eq:scrB-norm} to choose:
$$
\|(\Gamma,\Omega)\|=\sqrt{\|\Gamma\|_{\F}^2+\|\Omega\|_{\F}^2}
\quad\mbox{or}\quad
\|(\Gamma,\Omega)\|=\max\{\|\Gamma\|_{\F},\|\Omega\|_{\F}\}.
$$
We have the following corollary to \Cref{thm:SVD-almost}, where we state the conditions and the bounds on
$\|(\Gamma,\Omega)\|$ but omit the others that can be deduced from the bounds on $\|(\Gamma,\Omega)\|$.

\begin{corollary}\label{cor:SVD-almost-F}
Given $G,\,\wtd G\in\bbC^{m\times n}$, let $G$ be decomposed as in \eqref{eq:G:SVD:apx}, and partition
$U^{\HH}\wtd GV$ according to \eqref{eq:tG:apx-1}. Let  $\delta$,
$\munderbar\delta$,  and $\varepsilon$ be defined as in \eqref{eq:quantities-apx}.
\begin{enumerate}[{\rm (i)}]
  \item If
        \begin{equation}\label{eq:tG:small-pert-F1}
        \munderbar\delta>0
        \quad\mbox{and}\quad
        \kappa_2:=\frac {\varepsilon}{\munderbar\delta^2}\,\sqrt{\|E_{12}\|_{\F}^2+\|E_{21}\|_{\F}^2}<\frac 14,
        \end{equation}
        then \eqref{eq:GtG:syl':apx} has a solution $(\Gamma,\Omega)$ that satisfies
        \begin{align}
        \sqrt{\|\Gamma\|_{\F}^2+\|\Omega\|_{\F}^2}
         &\le\frac {1+\sqrt{1-4\kappa_2}}{1-2\kappa_2+\sqrt{1-4\kappa_2}}
                 \frac {\sqrt{\|E_{12}\|_{\F}^2+\|E_{21}\|_{\F}^2}}{\munderbar\delta} \nonumber \\
         &<2\,\frac {\sqrt{\|E_{12}\|_{\F}^2+\|E_{21}\|_{\F}^2}}{\munderbar\delta}. \label{eq:bd4SVDtG-F1}
        \end{align}
  \item If $\sigma_{\min}(G_1)>\sigma_{\max}(G_2)$ and if
        \begin{equation}\label{eq:tG:small-pert-F2}
        \munderbar\delta>0
        \quad\mbox{and}\quad
        \kappa_2:=\frac {\varepsilon}{\munderbar\delta^2}\,\max\{\|E_{12}\|_{\F},\|E_{21}\|_{\F}\}<\frac 14,
        \end{equation}
        then \eqref{eq:GtG:syl':apx} has a solution $(\Gamma,\Omega)$ that satisfies
        \begin{align}
        \max\{\|\Gamma\|_{\F},\|\Omega\|_{\F}\}
         &\le\frac {1+\sqrt{1-4\kappa_2}}{1-2\kappa_2+\sqrt{1-4\kappa_2}}
                 \frac {\max\{\|E_{12}\|_{\F},\|E_{21}\|_{\F}\}}{\munderbar\delta} \nonumber \\
         &<2\,\frac {\max\{\|E_{12}\|_{\F},\|E_{21}\|_{\F}\}}{\munderbar\delta}. \label{eq:bd4SVDtG-F2}
        \end{align}
\end{enumerate}
\end{corollary}

In comparing \Cref{cor:SVD-almost-F} with \Cref{thm:stew:1973SVD} \cite[Theorem~6.4]{stew:1973}, we note
\Cref{cor:SVD-almost-F}(i) provides better results than \Cref{thm:stew:1973SVD} in their conditions:
\eqref{eq:tG:small-pert-F1} {\em vs.} \eqref{eq:stew:cond}, and bounds: \eqref{eq:bd4SVDtG-F1} {\em vs.} \eqref{eq:stew:1973SVD-conc},
because
\begin{align*}
\hat\varepsilon^2=\|E_{12}\|_{\F}^2+\|E_{21}\|_{\F}^2
    &\ge\max\{\|E_{12}\|_2,\|E_{21}\|_2\}\sqrt{\|E_{12}\|_{\F}^2+\|E_{21}\|_{\F}^2} \\
    &=\varepsilon\sqrt{\|E_{12}\|_{\F}^2+\|E_{21}\|_{\F}^2}.
\end{align*}
Such an improvement of \Cref{cor:SVD-almost-F}(i) over \Cref{thm:stew:1973SVD} could be considered marginal because it can be easily recovered
by just refining Stewart's relevant estimates \cite{stew:1973}. However, the improvement by \Cref{cor:SVD-almost-F}(ii)
over \Cref{thm:stew:1973SVD}, under the condition $\sigma_{\min}(G_1)>\sigma_{\max}(G_2)$, unlikely can be achieved
by simply refining Stewart's arguments there.

Our original motivation to revisit  this classical result of Stewart's is the need for a version of \Cref{thm:stew:1973SVD}
in the spectral norm while we are working on an error analysis for model order reduction by  balanced truncation
\cite{zhli:2024}. Let us look at what \Cref{thm:SVD-almost} leads to for the spectral norm, for which the
two endowed norms from \eqref{eq:scrB-norm} collapse to the same one
$$
\|(\Gamma,\Omega)\|=\max\{\|\Gamma\|_2,\|\Omega\|_2\}.
$$
We have the following corollary to \Cref{thm:SVD-almost}, which yields far sharper results than those such as
\eqref{eq:stew:1973SVD-conc'} that would otherwise have to be derived from \Cref{thm:stew:1973SVD}.

\begin{corollary}\label{cor:SVD-almost-sp}
Given $G,\,\wtd G\in\bbC^{m\times n}$, let $G$ be decomposed as in \eqref{eq:G:SVD:apx}, and partition
$U^{\HH}\wtd GV$ according to \eqref{eq:tG:apx-1}. Let  $\delta$,
$\munderbar\delta$,  and $\varepsilon$ be defined as in \eqref{eq:quantities-apx}.
\begin{enumerate}[{\rm (i)}]
  \item If
        \begin{equation}\label{eq:tG:small-pert-sp1}
        \munderbar\delta>0
        \quad\mbox{and}\quad
        \kappa_2:=\left(\frac {\pi}2\right)^2\frac {\varepsilon^2}{\munderbar\delta^2}<\frac 14,
        \end{equation}
        then \eqref{eq:GtG:syl':apx} has a solution $(\Gamma,\Omega)$ that satisfies
        \begin{equation}\label{eq:bd4SVDtG-sp1}
        \max\{\|\Gamma\|_2,\|\Omega\|_2\}
         \le\frac {1+\sqrt{1-4\kappa_2}}{1-2\kappa_2+\sqrt{1-4\kappa_2}}
                 \frac {\pi}2\frac {\varepsilon}{\munderbar\delta}
         < {\pi}\,\frac {\varepsilon}{\munderbar\delta}.
        \end{equation}
  \item If $\sigma_{\min}(G_1)>\sigma_{\max}(G_2)$ and if
        \begin{equation}\label{eq:tG:small-pert-sp2}
        \munderbar\delta>0
        \quad\mbox{and}\quad
        \kappa_2:=\frac {\varepsilon^2}{\munderbar\delta^2}<\frac 14,
        \end{equation}
        then \eqref{eq:GtG:syl':apx} has a solution $(\Gamma,\Omega)$ that satisfies
        \begin{equation}\label{eq:bd4SVDtG-sp2}
        \max\{\|\Gamma\|_2,\|\Omega\|_2\}
         \le\frac {1+\sqrt{1-4\kappa_2}}{1-2\kappa_2+\sqrt{1-4\kappa_2}}
                 \frac {\varepsilon}{\munderbar\delta}
         <2\,\frac {\varepsilon}{\munderbar\delta}.
        \end{equation}
\end{enumerate}
\end{corollary}


During the proof of \Cref{lm:BdCpldSylv}, we commented that inequality \eqref{ineq:CpldSylv-UI-2} had been really implied
by Wedin \cite[section 3]{wedi:1972}, and inequality \eqref{ineq:stew-F} may also be essentially implied in \cite{wedi:1972} but with
his $\delta$
defined as
\begin{equation}\label{eq:delta-wedi1972}
\delta:=\min\Big\{\min_{\mu\in\sv(A),\,\nu\in\sv(B)}\,|\mu-\nu|,\,\sigma_{\min}(A)\Big\}
\end{equation}
which is the same as the one in \eqref{ineq:stew-F} for the case $s\ne t$ because then $0\in\sv_{\ext}(B)$,
but can be different if $s=t$ for which case $\sv_{\ext}(B)=\sv(B)$ that may or may not contain $0$, however.
Both \eqref{ineq:CpldSylv-UI-2} and  \eqref{ineq:stew-F} are used  to develop {\em the generalized $\sin\theta$ theorems\/} for SVD there
(see, also, \cite[p.21-7]{li:2014HLA}). In the same spirit, we can use \eqref{ineq:CpldSylv-UI-1}
and \eqref{ineq:CpldSylv-UI'} to establish more generalized $\sin\theta$ theorems for SVD. In fact, we have the following theorem.

\begin{theorem}\label{thm:more-sin(theta)-SVD}
Let $G\in\bbC^{m\times n}$ be decomposed as in \eqref{eq:G:SVD:apx}, and let $\wtd G\in\bbC^{m\times n}$ admit
a decomposition in the same form as in \eqref{eq:G:SVD:apx}, except with tildes on all symbols. Let
$$
R=G\wtd V_1-\wtd U_1\wtd G_1, \quad S=G^{\HH}\wtd U_1-\wtd V_1\wtd G_1^{\HH}.
$$
If
$$
\delta:=\min_{\mu\in\sv(\wtd G_1),\,\nu\in\sv_{\ext}(G_2)}\,|\mu-\nu|>0,
$$
then
        $$
        \left\|\begin{bmatrix}
          \sin\Theta(\cR(U_1),\cR(\wtd U_1)) & 0 \\
           0 & \sin\Theta(\cR(V_1),\cR(\wtd V_1))
        \end{bmatrix}\right\|_{\UI}\le c\,\frac 1{\delta}\left\|\begin{bmatrix}
                        R & 0 \\
                        0 & S
                      \end{bmatrix}\right\|_{\UI},
        $$
        where
$c={\pi}/2$ in general, but $c=1$ if also $\sigma_{\min}(\wtd G_1)>\sigma_{\max}(G_2)$ or for the Frobenius norm.
Here $\Theta(\cR(U_1),\cR(\wtd U_1))$ is the diagonal matrix of the canonical angles between the subspaces
        $\cR(U_1)$ and $\cR(\wtd U_1)$ {\rm \cite[p.21-2]{li:2014HLA}, \cite{stsu:1990}}.
\end{theorem}

The case for $\|\cdot\|_{\UI}=\|\cdot\|_{\F}$ is not new and Stewart and Sun \cite[p.260]{stsu:1990} credited it to Wedin \cite{wedi:1972}
with a slightly
different $\delta$ (similar to \eqref{eq:delta-wedi1972} we commented moments ago). However, Wedin \cite{wedi:1972} did not
explicitly mention it for the case. Evidently, Wedin could easily had it because of the machinery  he had already built in the paper.

\begin{proof}[Proof of \Cref{thm:more-sin(theta)-SVD}]
It can be seen that
$$
U_2^{\HH}R = G_2 V_2^{\HH}\wtd V_1-U_2^{\HH}\wtd U_1\wtd G_1, \quad
V_2^{\HH}S = G_2^{\HH}U_2^{\HH}\wtd U_1-V_2^{\HH}\wtd V_1\wtd G_1^{\HH}.
$$
Or, equivalently,
$$
(U_2^{\HH}\wtd U_1)\wtd G_1-G_2 (V_2^{\HH}\wtd V_1)=-U_2^{\HH}R, \quad
(V_2^{\HH}\wtd V_1)\wtd G_1^{\HH}-G_2^{\HH}(U_2^{\HH}\wtd U_1)=-V_2^{\HH}S,
$$
which takes the form of the coupled Sylvester equations \eqref{eq:CpldSyvl} with $X=U_2^{\HH}\wtd U_1$ and $Y=V_2^{\HH}\wtd V_1$.
Noticing that the singular values of $U_2^{\HH}\wtd U_1$ and those of $V_2^{\HH}\wtd V_1$ are the same as the sines of the canonical angles
between $\cR(U_1)$ and $\cR(\wtd U_1)$ and between $\cR(V_1)$ and $\cR(\wtd V_1)$, respectively, and hence \cite{stsu:1990,li:1993d,li:1994a}
$$
\left\|\begin{bmatrix}
          U_2^{\HH}\wtd U_1 & 0 \\
           0 & V_2^{\HH}\wtd V_1
        \end{bmatrix}\right\|_{\UI}=
\left\|\begin{bmatrix}
          \sin\Theta(\cR(U_1),\cR(\wtd U_1)) & 0 \\
           0 & \sin\Theta(\cR(V_1),\cR(\wtd V_1))
        \end{bmatrix}\right\|_{\UI},
$$
and noticing
$$
\left\|\begin{bmatrix}
          -U_2^{\HH}R & 0 \\
           0 & -V_2^{\HH}S
        \end{bmatrix}\right\|_{\UI}
=\left\|-\begin{bmatrix}
          U_2^{\HH} & 0 \\
           0 & V_2^{\HH}
        \end{bmatrix}
\begin{bmatrix}
          R & 0 \\
           0 & S
        \end{bmatrix}\right\|_{\UI}
\le\frac 1{\delta}\left\|\begin{bmatrix}
                        R & 0 \\
                        0 & S
                      \end{bmatrix}\right\|_{\UI},
$$
the conclusion in the theorem is a simple consequence of \Cref{lm:BdCpldSylv}.
\end{proof}

\section{Concluding Remarks}\label{sec:concl}
Stewart's original theorem \cite[Theorem~6.4]{stew:1973} for the singular subspaces associated with the SVD of a matrix
subject to some perturbations uses both the Frobenius and spectral norms. Although it is still versatile to apply in
the situations such as only the spectral norm is suitable \cite{zhli:2024}, it may lead
to weaker results: stronger conditions and yet less sharp bounds, through equivalency bounds between the spectral norm and the Frobenius norm,
as we argued in \Cref{sec:intro}. Consequently it pays to establish variants of Stewart's original theorem
from scratch. Our main contribution in \Cref{thm:SVD-almost} provides perturbation bounds that encompass a variety of circumstances
and that are dictated by the conditions as specified by \Cref{lm:veri-stew-bT}, which are further explained in \Cref{rk:veri-stew-bT}.

\Cref{lm:BdCpldSylv} collects bounds, some old and some new, on the solution to
the set of coupled Sylvester equations \eqref{eq:CpldSyvl} under different circumstances. They form a part of the foundation
based on which our main results in \Cref{thm:SVD-almost} are derived. Furthermore \Cref{lm:BdCpldSylv} can be put into good use to
establish more generalized $\sin\theta$ theorems for SVD as exemplified in \Cref{thm:more-sin(theta)-SVD}, beyond existing ones due to Wedin~\cite{wedi:1972}.

\def\noopsort#1{}\def\l{\char32l}\def\v#1{{\accent20 #1}}
  \let\^^_=\v\def\hbk{hardback}\def\pbk{paperback}
  \providecommand{\href}[2]{#2}
\providecommand{\arxiv}[1]{\href{http://arxiv.org/abs/#1}{arXiv:#1}}
\providecommand{\url}[1]{\texttt{#1}}
\providecommand{\urlprefix}{URL }


\end{document}